\documentclass[11pt]{amsart}

\usepackage{amsmath,amsthm,amsfonts,amssymb,bm,graphicx }

\makeatletter

\usepackage
{hyperref}
\hypersetup{colorlinks=true,citecolor=blue,linkcolor=blue,urlcolor=blue}

\theoremstyle{plain}
\newtheorem{theorem}{Theorem}
\newtheorem{lemma}[theorem]{Lemma}

\newtheorem{prop}[theorem]{Proposition}
\theoremstyle{definition}

\newtheorem{example}[theorem]{Example}

\numberwithin{equation}{section}

\newcommand{\N}{\mathbb{N}}
\newcommand{\Z}{\mathbb{Z}}
\newcommand{\Q}{\mathbb{Q}}

\newcommand{\A}{\alpha}
\newcommand{\B}{\beta}
\newcommand{\G}{\gamma}
\newcommand{\Tr}{\mathrm{Tr}}

\makeatletter
\DeclareRobustCommand\widecheck[1]{{\mathpalette\@widecheck{#1}}}
\def\@widecheck#1#2{%
    \setbox\z@\hbox{\m@th$#1#2$}%
    \setbox\tw@\hbox{\m@th$#1%
       \widehat{%
          \vrule\@width\z@\@height\ht\z@
          \vrule\@height\z@\@width\wd\z@}$}%
    \dp\tw@-\ht\z@
    \@tempdima\ht\z@ \advance\@tempdima2\ht\tw@ \divide\@tempdima\thr@@
    \setbox\tw@\hbox{%
       \raise\@tempdima\hbox{\scalebox{1}[-1]{\lower\@tempdima\box
\tw@}}}%
    {\ooalign{\box\tw@ \cr \box\z@}}}
\makeatother

\begin{document}

\title
{Universal quadratic forms over multiquadratic fields}

\author{V\'\i t\v ezslav Kala}
\address{Charles University, Faculty of Mathematics and Physics, Department of Algebra, Sokolov\-sk\' a 83, 18600 Praha~8, Czech Republic}
\address{University of Goettingen, Mathematisches Institut, Bunsenstr.~3-5, D-37073 G\"ottingen, Germany}
\email{vita.kala@gmail.com}

\author{Josef Svoboda}
\address{Charles University, Faculty of Mathematics and Physics, Department of Algebra, Sokolov\-sk\' a 83, 18600 Praha~8, Czech Republic}
\email{josefsvobod@gmail.com}

\thanks{V. K. was supported by Czech Science Foundation GA\v CR, grant 17-04703Y. J. S. was supported by the project SVV-2017-260456}


\date{\today}


\begin{abstract}
For all positive integers $k$ and $N$ we prove that there are infinitely many totally real multiquadratic fields $K$ of degree $2^k$ over $\Q$ such that each
universal quadratic form over $K$ has at least $N$ variables.
\end{abstract}

\maketitle

\section{Introduction}

The arithmetics of quadratic forms has a colorful history, and we will be interested in one of its important aspects, \emph{universality}:
a totally positive definite quadratic form with coefficients in the ring of integers $\mathcal O_K$ of a totally real number field $K$
is universal if it represents all totally positive integers. 
To slightly simplify our discussion, let us denote by $m(K)$ the minimal arity of such a form.
Starting with the classical fact that $m(\mathbb Q)=4$ (e.g., the sum of 4 squares is universal), the theory of universal forms over $\mathbb Z$ was essentially completed by the 15- and 290- theorems \cite{Bh}, \cite{BH}. 

This was followed by a number of results concerning forms over real quadratic fields: besides from some interesting theorems over specific fields of small discriminant obtained by various authors \cite{CKR}, \cite{De}, \cite{Sa},
Kim \cite{Ki}, \cite{Ki2} showed that $\mathbb Q(\sqrt {n^2-1})$ always has an 8-ary universal form and also proved that there are only finitely many values of $D$ such that $\Q(\sqrt D)$ has a {diagonal} 7-ary universal quadratic form.
Blomer and Kala \cite{BK}, \cite{Ka} recently used continued fractions to construct, for each $N$, infinite families of real quadratic fields such that $m(\mathbb Q(\sqrt D))\geq N$ in these families.

In general, already in 1945 Siegel \cite{Si} proved that $\mathbb Q$ and $\mathbb Q(\sqrt 5)$ are the only number fields (of any degree) over which the sum of any number of squares is universal (over $\mathcal O_K$). This changes dramatically when one enlarges $\mathcal O_K$: recently Collinet \cite{Co} showed that the sum of five squares is universal over $\mathcal O_K\left[\frac 12\right ]$ for any number field $K$.

Otherwise, not much is known beyond the quadratic case. One of the reasons why the higher degree case is harder is that the theory of continued fractions
and their convergents is not available, and hence we do not understand the arithmetics of $\mathcal O_K$ well enough.

Nevertheless, the goal of this short note is to generalize the aforementioned result of Blomer and Kala to the case of multiquadratic fields, i.e., number fields of the form $K=\mathbb Q(\sqrt {p_1}, \sqrt {p_2}, \dots, \sqrt{p_k})$ with $k\geq 1$ and $p_1, p_2, \dots, p_k\in\mathbb N$.

\begin{theorem}\label{main thm}
For all pairs of positive integers $k$, $N$ there are infinitely many totally real multiquadratic fields $K$ of degree $2^k$ over $\Q$ such that 
$m(K)\geq N$.
\end{theorem}

As far as we know, this is the first general result over number fields of degree greater than two. Our general approach is the same as in \cite{Ka}, where
one produces sufficiently many additively indecomposable integers (with suitable additional properties, see Proposition \ref{promen} below). 
We proceed by induction on $k$, starting with a quadratic subfield  $\mathbb Q(\sqrt {p_1})$ of $K$ in which we can find the desired elements $a_i$. Then  we inductively choose sufficiently large integers $p_2, \dots, p_k$ and show that our elements $a_i$ retain their necessary properties. For this we need some
basic results on the arithmetics of totally positive integers in $K$, which we collect in Section \ref{sec 2}. 
We first review them in the biquadratic case which is notationally simpler and only then give the general statements.

An interesting question is whether it is possible
to generalize our arguments to other number fields $K$.
One could perhaps hope to push them as far as general totally real fields of even degree over $\mathbb Q$, but this seems not to be easy.  
Anyway, as we rely on the existence of a quadratic subfield of $K$, the situation of odd degree fields will certainly require significant new ideas.

\section{Totally positive integers}\label{sec 2}

For the study of universal quadratic forms, we first need to recall some basics about number fields, mostly concerning totally positive integers and indecomposables.

Let $K$ be a totally real number field of degree $d$ over $\mathbb Q$, i.e., $K$ has exactly $d$ real embeddings $\sigma: K\hookrightarrow\mathbb R$. 
An element $\alpha\in K$ is \emph{totally positive} if $\sigma(\alpha)>0$ for all $\sigma$. We'll write $\alpha\succ\beta$ if $\alpha-\beta$ is totally positive, and $\alpha\succeq\beta$ if $\alpha=\beta$ or $\alpha\succ\beta$; the semiring of all totally positive algebraic integers is denoted by $\mathcal{O}_K^+$.
Finally, a totally positive integer $\alpha\in\mathcal O_K^+$ is \emph{(additively) indecomposable} if it can't be decomposed as the sum of two totally positive integers.

Let's also make precise our definitions concerning quadratic forms over $K$. 
Such an $N$-ary form is defined as $Q(x_1, \dots, x_N)=\sum_{1\leq i\leq j\leq N} a_{ij}x_ix_j$ with coefficients $a_{ij}\in\mathcal O_K$.
We say that $Q$ is \emph{totally positive definite} if $\sigma(Q)(x_1, \dots, x_N):=\sum_{1\leq i\leq j\leq N}\sigma(a_{ij})x_ix_j$ is a positive definite quadratic form for each $\sigma: K\hookrightarrow\mathbb R$. And a form $Q$ is \emph{universal} if it represents all totally positive integers, i.e., 
for each $\alpha\in\mathcal O_K^+$ there are $x_1, \dots, x_N\in\mathcal O_K$ with $Q(x_1, \dots, x_N)=\alpha$.

\subsection{Biquadratic fields}\label{subsec biquad}

In the case of biquadratic fields, the notation and results are clearer, so we'll start with it as an example. 
Hence in this subsection, let $K=\mathbb Q(\sqrt p, \sqrt q)$ be a totally real biquadratic number field, i.e., $p$ and $q$ are squarefree positive integers such that $K$ has degree 4 over $\mathbb Q$. Define $$r:=\frac{pq}{\gcd(p,q)^2};$$ for an element $\alpha=a+ b\sqrt p+c\sqrt q+d\sqrt r\in K$ with $a, b, c, d\in\Q$ we order its conjugates as

$\alpha^{(1)}=a+ b\sqrt p+c\sqrt q+d\sqrt r,$

$\alpha^{(2)}=a+ b\sqrt p-c\sqrt q-d\sqrt r,$ 

$\alpha^{(3)}=a- b\sqrt p+c\sqrt q-d\sqrt r,$

$\alpha^{(4)}=a- b\sqrt p-c\sqrt q+d\sqrt r.$

The ring of integers $\mathcal O_K$ can be explicitly described in terms of the values of $p,q,r$ modulo 4. The precise description is a little technical; we will only need to know that if $\alpha\in\mathcal O_K$, then its coefficients always lie in $\frac{1}{4}\mathbb Z$.

\begin{example}\label{baze}
In the case $p\equiv 3\pmod 4, q\equiv r\equiv 2\pmod 4$, an integral basis for $\mathcal O_K$ can be chosen of the form 
$1, \sqrt p, \sqrt q, \frac{\sqrt q+\sqrt r}2$; when $p\equiv q\equiv r\equiv 1\pmod 4$, one can take 
$1, \frac{1+\sqrt p}2, \frac{1+\sqrt q}2, \frac{(1+\sqrt p)(1+\sqrt q)}4$ as a basis -- see, e.g., \cite[Proposition 8.22]{Jarvis}.
\end{example}

We will later need to know that the trace of a totally positive integer can't be too small:

\begin{lemma}\label{stopa} 
Suppose that $\A = a+b\sqrt{p}+c\sqrt{q}+d\sqrt{r} \in \mathcal{O}_K^+$ (with $a, b, c, d\in\Q$). Let $\Tr$ denote the trace from $K$ to $\Q$.
\begin{enumerate}
\item If $b \neq 0$, then
$\Tr(\A) > \sqrt{p}$.
\item If $c \neq 0$, then
$\Tr(\A) > \sqrt{q}$.
\item If $d \neq 0$, then
$\Tr(\A) > \sqrt{r}$.
\end{enumerate}
Finally, if $\A \notin \mathbb{Z}$, then $\Tr(\A)> \min(\sqrt{p},\sqrt{q},\sqrt{r}).$
\end{lemma}

\begin{proof}
All the cases are analogous, so let's prove only the one when $b \neq 0$. Since $\A \in \mathcal{O}_K$, the coefficients $a, b, c, d$ are quarter-integers. The~element $\A$ is totally positive, and so $\alpha^{(1)}+\alpha^{(2)}>0$. Hence $a>-b\sqrt{p}$. Similarly, from $\alpha^{(3)}+\alpha^{(4)}>0$ it follows that $a>b\sqrt{p}$.
Since $b$ is a~nonzero quarter-integer, the~inequality $a>|b\sqrt{p}|$ implies that $a>\frac{\sqrt{p}}{4}$, and hence $\Tr(\A)=4a >\sqrt{p}.$
\end{proof}

Since we know integral bases explicitly, we can sometimes sharpen the estimates (although we won't need this in the paper).  E.g., in the first case from Example \ref{baze}, we have $\Tr(\A)> \min(4\sqrt{p},2\sqrt{q},2\sqrt{r})$ if $\alpha\neq 0$. 

Similarly as in the quadratic case \cite[Lemma 3]{BK}, the previous lemma implies that all elements of sufficiently small norm are indecomposable:

\begin{prop}	\label{normab}
Assume that $\A \in \mathcal{O}_K^+$ satisfies $N(\A)<2\min(\sqrt{p},\sqrt{q},\sqrt{r})$ and  $n \nmid \A$ for~every $n \in \mathbb{N}, n>1$. Then $\A$ is indecomposable.  
\end{prop}

\begin{proof}
Denote by $\delta:=\min(\sqrt{p},\sqrt{q},\sqrt{r})$. For~contradiction, suppose that $\A=\B+\G$ where $\B$, $\G \in \mathcal{O}_K^+$. Then we have 

\begin{align}
2\delta &> N(\A) = \left(\B^{(1)}+\G^{(1)}\right)\left(\B^{(2)}+\G^{(2)}\right)\left(\B^{(3)}+\G^{(3)}\right)\left(\B^{(4)}+\G^{(4)}\right)> \nonumber \\
&>
\Tr\left(\B^{(1)}\G^{(2)}\G^{(3)}\G^{(4)}\right)+
\Tr\left(\B^{(1)}\B^{(2)}\B^{(3)}\G^{(4)}\right) + \text{other (positive) summands}. \nonumber
\end{align}

\noindent The final part of Lemma \ref{stopa} implies that $\B^{(1)}\G^{(2)}\G^{(3)}\G^{(4)}$ or $\B^{(1)}\B^{(2)}\B^{(3)}\G^{(4)}$ must be a (positive) integer. 
Without loss of generality, suppose that it is the former. This element then equals each of its conjugates and we have
$$\B^{(1)}\G^{(2)}\G^{(3)}\G^{(4)}=\B^{(2)}\G^{(1)}\G^{(4)}\G^{(3)} \in \mathbb{N}.$$

\noindent If we divide the~equality by~the~norm of~$\G$ and multiply it by $\G^{(1)}$, we get
$$\B=\B^{(1)}=\frac{\B^{(2)}\G^{(1)}\G^{(4)}\G^{(3)}}{N(\G)} \G^{(1)} =\frac{u}{v} \G,$$
where $u$ and $v$ are coprime natural numbers. Then we have $\B=u\mu$ and $\G = v\mu$ for $\mu:=\G/v \in \mathcal{O}_K$, so $\A = (u+v)\mu$
is divisible by the integer $u+v\geq 2$, a contradiction.
\end{proof}

Let us briefly remark that much more is known about the sizes of norms of indecomposables in the quadratic case (thanks to their description in terms of continued fractions) -- see, e.g., \cite{JK}, \cite{Ka2}.

\subsection{Multiquadratic fields}

Let us now turn our attention to general multiquadratic fields. In this subsection, take $k\in\N$ and let $p_1, \dots, p_k$ be squarefree positive integers such that $K:=\Q(\sqrt {p_1}, \dots, \sqrt {p_k})$
has degree $2^k$ over $\Q$.
To describe the Galois group of $K$ and its ring of integers $\mathcal O_K$, for $I\subset\{1, \dots, k\}$ we denote by $p_I$ the squarefree part of $\prod_{i\in I}p_i$, i.e., $p_I=\frac 1{\ell^2}\prod_{i\in I}p_i$, where $\ell$ is the largest integer such that the right hand side is an integer.
Clearly, $\{\sqrt{p_I}| I\subset\{1, \dots, k\}\}$ is a $\Q$-vector space basis of $K$ and $\Z[\sqrt{p_I} | I\subset\{1, \dots, k\}]\subset\mathcal O_K$.

Let $\sigma_i:K\rightarrow K$ be the automorphism of $K$ defined by $\sigma_i(p_j)=(-1)^{\delta_{ij}}p_j$, where $\delta_{ij}$ is the Kronecker delta.
For $I\subset\{1, \dots, k\}$ let $\sigma_I:=\prod_{i\in I}\sigma_i$, so that $\mathrm{Gal}(K/\Q)=\{\sigma_I|I\subset\{1, \dots, k\}\}$.

Note that in the notation of Subsection \ref{subsec biquad}, we have $\alpha^{(1)}=\sigma_{\emptyset}(\alpha)$, 
$\alpha^{(2)}=\sigma_{2}(\alpha)$, $\alpha^{(3)}=\sigma_{1}(\alpha)$, and $\alpha^{(4)}=\sigma_{12}(\alpha)$.

\medskip

It is possible to describe a $\Z$-basis for $\mathcal O_K$ explicitly \cite{Sch}; we will only need to know that 
\begin{equation}\label{eq ok}
\mathcal O_K\subset \frac 1{2^k}\Z\left[\sqrt{p_I} | I\subset\{1, \dots, k\}\right].
\end{equation}
 This follows from the proof of Satz 3.2 in \cite{Sch}, 
but one can also prove it easily by induction on $k$ in the case when $p_i$ are pairwise coprime (so that $p_I=\prod_{i\in I}p_i$), as we now sketch:

First, at most one of $p_i$ is even, and so we can reorder them so that $p_2, \dots, p_k$ are odd.
Denote by $K_i:=\Q(\sqrt {p_1}, \dots, \sqrt {p_i})$. We'll prove the claim by induction on $i$; for the quadratic field $K_1$ it's well-known. Hence it remains to show that $\mathcal O_{K_i}\subset \frac 12 \mathcal O_{K_{i-1}}[\sqrt{p_i}]$ for $i\geq 2$. 
Let $\frac a2+\frac b2\sqrt {p_i}\in \mathcal O_{K_i}$ with $a, b\in K_{i-1}$.
Its relative trace from $K_i$ to $K_{i-1}$ is $a$, and so $a\in \mathcal O_{K_{i-1}}$.
Considering its relative norm then implies that $b^2p_i\in \mathcal O_{K_{i-1}}$. By our assumptions on $p_i$, none of its prime factors ramifies in $K_{i-1}$, and so the prime ideal factorization of $b^2p_i\mathcal O_{K_{i-1}}$ implies that $b\in \mathcal O_{K_{i-1}}$, as required.

\medskip

Finally, we will need a generalization of Lemma \ref{stopa}.

\begin{lemma}\label{trace}
Suppose that $\alpha=\sum_{I\subset\{1, \dots, k\}}a_I\sqrt{p_I}\in\mathcal O_K^+$ (with $a_I\in\Q$).
If $a_I\neq 0$, then $\Tr_{K/\Q}(\alpha)>\sqrt {p_I}$.
\end{lemma}

\begin{proof}
We have 
$$\sum_{J,\ \# I\cap J\mathrm{\ is\ even}}\sigma_J(\alpha)=2^{k-1}(a_\emptyset+a_I\sqrt{p_I})>0,$$
because $\alpha$ is totally positive, and similarly 
$$\sum_{J,\ \# I\cap J\mathrm{\ is\ odd}}\sigma_J(\alpha)=2^{k-1}(a_\emptyset-a_I\sqrt{p_I})>0.$$
Putting these two inequalities together, we see that $a_\emptyset>|a_I|\sqrt{p_I}\geq\frac 1{2^k}\sqrt{p_I}$, 
as $0\neq a_I\in\frac 1{2^k}\Z$.
This implies that $\Tr_{K/\Q}(\alpha)=2^ka_\emptyset>\sqrt {p_I}$.
\end{proof}

\section{Proof of Theorem 1}\label{sec 3}

We are now ready to prove Theorem \ref{main thm}. As in the quadratic case, our approach is based on the following proposition:

\begin{prop}[\cite{Ka}, Proposition 2.1]\label{promen}
Assume that $K$ is a totally real number field such that there are elements $a_1,a_2,\dots,a_N\in\mathcal O_K^+$ such~that for~all $1 \leq i< j \leq N$ we have: if $c \in \mathcal{O}_K$ and $4a_ia_j \succeq c^2$, then $c=0.$ Then there are no universal totally positive $(N-1)$-ary quadratic forms over~$\mathcal{O}_K$.
\end{prop}

We shall not repeat its proof here, but the rough idea is that the inequalities (essentially) force each universal form over $K$ to contain the diagonal subform $a_1x_1^2+\dots+a_Nx_N^2$. 
Note that the proposition covers arbitrary integral forms without assuming that they are classical (in the sense of having all cross-coefficients divisible by 2).

Blomer and Kala then used convergents to the continued fraction for $\sqrt D$ to construct the elements $a_1,a_2,\dots,a_N$, and hence to prove the following theorem.

\begin{theorem}[\cite{BK}, Theorem 1, and \cite{Ka}, Theorem 1.1]\label{forms} 
For every $N$ there exist infinitely many squarefree positive integers $D$ such that $K=\Q(\sqrt D)$ contains elements $a_1,a_2,\dots a_N$ from~Proposition \ref{promen}. Consequently,
every universal (totally positive) quadratic form  over $\Q(\sqrt{D})$ has at least $N$ variables, i.e., $m(\mathbb Q(\sqrt D))\geq N$.
\end{theorem}

As we already indicated in the introduction, we will now use these results to prove Theorem \ref{main thm}.

\begin{proof}[Proof of Theorem \ref{main thm}.]
To prove the theorem for a multiquadratic field $\Q(\sqrt {p_1}, \dots, \sqrt {p_k})$, we will proceed by induction on $k$. Theorem \ref{forms} gives the desired result in the quadratic case when $k=1$.

For the induction step, assume that $p_1, \dots, p_{k}$ are pairwise coprime squarefree positive integers, so that $K:=\Q(\sqrt {p_1}, \dots, \sqrt {p_{k}})$ has degree $2^{k}$ and satisfies the assumptions of Proposition \ref{promen}, i.e., there are elements $a_1,a_2,\dots a_N\in \mathcal{O}_{K}^+$ 
such~that for~all $1 \leq i< j \leq N$ we have that if $c \in \mathcal{O}_{K}$ and $4a_ia_j \succeq c^2$, then $c=0.$

The idea of~our proof is to find some squarefree $q=p_{k+1}$, coprime to all $p_i$, such that in~$L:=K(\sqrt q)$, the~elements $a_i$ still satisfy the~condition of~Proposition \ref{promen}. We will use our knowledge (\ref{eq ok}) about the rings of integers, 
combined with Lemma \ref{stopa} which gives a lower bound on the trace of positive integers. In the proof, $\Tr_K$ denotes the trace from $K$ to $\Q$ and 
$\Tr_L$ is the trace from $L$ to $\Q$ (note that if $\alpha\in K\subset L$, then $\Tr_L(\A)=2\Tr_K(\A)$).

Namely, choose a~squarefree~positive integer $q=p_{k+1}$ such that
\begin{itemize}
\item $\sqrt{q}>2^{k+1}$, 
\item $\sqrt{q} >8\Tr_K(a_ia_j)$ for~every $1 \leq i <j\leq N$, and
\item $\gcd(p_i,q)=1$ for each $i=1, \dots, k$.
\end{itemize}

To verify the assumptions of Proposition \ref{promen} for $a_1,a_2,\dots a_N\in L=K(\sqrt q)$, suppose that $4a_ia_j \succeq c^2$ for~some indices $1\leq i< j\leq N$ and $c \in \mathcal{O}_L$. This in particular implies that $4\Tr_L(a_ia_j)\geq\Tr_L(c^2)$.

Let $c=u +v\sqrt{q}$ where $u,v \in K$, and let's distinguish three cases: 

a) $v=0$. Then $c\in K$, and in fact $c\in\mathcal O_K=K\cap\mathcal O_L$.
Then our assumption on the elements $a_i$ implies that $c=0$.

b) $u=0, v\neq 0$. Then $c^2 = v^2q$ and we have
$$\Tr_L(c^2)=q \Tr_L(v^2)=2q\Tr_K(v^2) \geq \frac 1{2^{k+1}} q > \sqrt{q}.$$
Here the last inequality holds by the first assumption on $q$, and
the inequality $\Tr_K(v^2)\geq \frac 1{2^{k+2}}$ holds since by (\ref{eq ok}), $v\sqrt q\in\mathcal O_L$ implies that $v\in \frac 1{2^{k+1}}\Z[\sqrt{p_I} | I\subset\{1, \dots, k\}]$, and then 
$v^2\in \frac 1{2^{2k+2}}\Z[\sqrt{p_I} | I\subset\{1, \dots, k\}]$, and is totally positive. 
 
But then we have contradiction with $\sqrt{q} > 8\Tr_K(a_ia_j)=4\Tr_L(a_ia_j)\geq\Tr_L(c^2)$. 

c) $u, v\neq 0$. Then $$c^2=u^2+v^2q+2uv\sqrt{q}=:\sum_{I\subset\{1, \dots, k+1\}}a_I\sqrt{p_I}\in\mathcal O_L^+.$$ 

Since $2uv$ is nonzero, at least one of the coefficients $a_I$ with $k+1\in I$ is nonzero, and so Lemma \ref{trace} implies that
$\Tr_L(c^2)>\min(\sqrt{p_I}|k+1\in I)=\sqrt{q}$ (this equality holds because $p_I=\prod_{i\in I} p_i$ and $p_{k+1}=q$).
But then we have
$$8\Tr_K(a_ia_j)= 4\Tr_L(a_ia_j) \geq \Tr_L(c^2) > \sqrt{q},$$ which contradicts the second condition of the definition of $q$. 

We have proved that $c=0$, and so the~elements $a_i$ satisfy the~condition of Proposition \ref{promen} (over $L$),
concluding the proof.
\end{proof}


\begin{thebibliography}{CKR}

\bibitem[Bh]{Bh} M. Bhargava, \emph{On the Conway-Schneeberger Fifteen Theorem},  Contemp. Math. \textbf{272} (1999), 27-37. 

\bibitem[BH]{BH} M. Bhargava, J. Hanke, \emph{Universal quadratic forms and the 290-theorem}, Invent. Math., to appear.




\bibitem[BK]{BK} V. Blomer, V. Kala, \emph{Number fields without universal $n$-ary quadratic forms}, Math. Proc. Cambridge Philos. Soc. \textbf{159} (2015), 239-252.

\bibitem[CKR]{CKR} W. K. Chan, M.-H. Kim, S.  Raghavan, \emph{Ternary universal integral quadratic forms}, Japan. J. Math. \textbf{22} (1996), 263-273. 

\bibitem[Co]{Co} G. Collinet, \textit{Sums of squares in rings of integers with 2 inverted}, Acta Arith. \textbf{173} (2016), no. 4, 383-390.

\bibitem[De]{De} J. I. Deutsch, \emph{Universality of a non-classical integral quadratic form over $\mathbb Q[\sqrt 5]$}, Acta Arith. \textbf{136} (2009), 229-242.









\bibitem[JK]{JK} S. W. Jang, B. M. Kim, \emph{A refinement of the Dress-Scharlau theorem}, J. Number Theory \textbf{158} (2016), 234-243.

\bibitem[Ja]{Jarvis} F. Jarvis, \emph{Algebraic number theory}, Springer (2014).

\bibitem[Ka1]{Ka} V. Kala, \emph{Universal quadratic forms and elements of small norm in real quadratic fields}, Bull. Aust. Math. Soc. \textbf{94} (2016), 7-14.

\bibitem[Ka2]{Ka2} V. Kala, \emph{Norms of indecomposable integers in real quadratic fields}, J. Number Theory \textbf{166} (2016), 193-207.


\bibitem[Ki1]{Ki} B. M.  Kim, \emph{Finiteness of real quadratic fields which admit positive integral diagonal septenary universal forms}, Manuscr. Math. \textbf{99} (1999), 181-184.

\bibitem[Ki2]{Ki2} B. M. Kim, \emph{Universal octonary diagonal forms over some real quadratic fields}, Commentarii Math. Helv. \textbf{75} (2000), 410-414.














\bibitem[Sa]{Sa} H. Sasaki, \emph{Quaternary universal forms over $\mathbb Q[\sqrt{13}]$}, Ramanujan J. \textbf{18} (2009), 73-80.



\bibitem[Sch]{Sch} B. Schmal, \emph{Diskriminanten, $\Z$-Ganzheitsbasen und relative Ganzheitsbasen
bei multiquadratischen Zahlk\" orpern}, Arch. Math. \textbf{52} (1989), 245-257.




\bibitem[Si]{Si} C. L. Siegel, \emph{Sums of $m$-th powers of algebraic integers},  Ann. Math. \textbf{46} (1945),  313-339.



\end{thebibliography}
\end{document}